\documentclass[12pt]{amsart}
\usepackage[latin1]{inputenc}
\usepackage{amssymb, amsmath, amscd, mathrsfs,marvosym, psfrag,color,a4} 

\input xy
\xyoption{all}
\DeclareMathAlphabet{\mathpzc}{OT1}{pzc}{m}{it}

\newtheorem{theorem}{Theorem}[section]

\newtheorem{corollary}[theorem]{Corollary}

\newtheorem{lemma}[theorem]{Lemma}

\theoremstyle{definition}

\theoremstyle{remark}
\newtheorem{remark}[theorem]{Remark}

\begin{document}

\pagenumbering{arabic}
\title[]{Semi-infinite combinatorics in representation theory} \author[]{Martina Lanini}
\begin{abstract} In this work we discuss some appearances of semi-infinite combinatorics in representation theory. We propose a semi-infinite moment graph theory and we motivate it by considering the (not yet rigorously defined) geometric side of the story. We show that it is possible to compute stalks of the local intersection cohomology of the semi-infinite flag variety, and hence of spaces of quasi maps, by performing an  algorithm due to Braden and MacPherson.
\end{abstract}

\address{School of Mathematics, University of Edinburgh, Edinburgh, EH9 3FD, UK}
\email{m.lanini@ed.ac.uk}
\maketitle

\section{Introduction}

Semi-infinite combinatorics occurs - or it is expected to occur- in representation theory of quantum groups at a root
of unity, of Lie algebras and algebraic groups in positive characteristics, of complex affine Kac-Moody algebras.  
This paper does not aim at furnishing an exhaustive list of such occurrences, but rather at presenting the ones
the author has personally been working on.

While looking for evidences for his modular conjecture, Lusztig introduced the periodic module, a certain 
$\mathbb{Z}[v^{\pm 1}]$-module equipped with an action of the affine Hecke algebra \cite{Lu2}. In loc. cit. 
he studied a periodic analogue of Kazhdan-Lusztig basis elements and in his 1990 ICM paper \cite{Lu1} he related
them to the geometry of the \emph{periodic Schubert varieties}: Lusztig's generic polynomials were declared to play 
in this context the role played by Kazhdan-Lusztig polynomials in the usual Schubert variety setting. Such a fact was
stated in \cite{Lu2} without a proof. A more precise statement (and its proof) involving Drinfeld's spaces of quasi
maps was given in \cite{FFKM}.

In the same year, independently, Feigin and Frenkel also considered these varieties, using the denomination
\emph{semi-infinite}, which seems to be now the preferred name in the literature. Their motivation to introduce 
the semi-infinite flag variety was a  geometric construction of a class of modules attached to any affine Kac-Moody
algebra, the Wakimoto modules, which had been previously considered in the $\widehat{\mathfrak{sl}_2}$-case by 
Wakimoto \cite{Wak}.  Wakimoto modules are  realised in \cite{FF} as zero extensions of constant sheaves on semi-infinite
Schubert cells. On the other hand, under the Beilinson-Bernstein and  Riemann-Hilbert correspondences, (dual) Verma
modules can be obtained as zero-extensions of constant sheaves on Schubert cells and we are hence allowed to think of
Wakimoto modules as the semi-infinite analogue of (dual) Verma modules.

We have already mentioned that Lusztig investigated in \cite{Lu2} \emph{a} periodic analogue of Kazhdan-Lusztig
basis elements. In fact, he studied \emph{two} families of polynomials appearing in this periodic context and
both of them could  be considered as a periodic analogue of Kazhdan-Lusztig polynomials: Lusztig's generic polynomials,
which have already appeared in this introduction, and Lusztig's periodic polynomials. We will recall their definition
in Section \ref{Sec_HeckeMods}. In the same way as Kazhdan-Lusztig polynomials control  Jordan-H\"older multiplicities
of Verma modules, also periodic and generic polynomials (are expected to) govern Jordan-H\"older multiplicities of 
certain standard objects in various representation categories (see \S\ref{Sec_Formulae}).

The idea of using moment graph techniques for solving multiplicity formula problems is due to Fiebig and it is 
motivated by Soergel's approach to the Kazhdan-Lusztig conjecture on the characters of irreducible modules for 
finite dimensional complex Lie algebras. The main point of such a strategy is to combinatorially describe intersection
cohomology groups of Schubert varieties. This can be done by either using the theory of Soergel bimodules or of 
Braden-MacPherson sheaves on Bruhat graphs.

Here we  show that the theory of Braden-MacPherson sheaves on semi-infinite graphs can be applied to calculate
local intersection cohomology of semi-infinite Schubert varieties (once made sense of them). This is the only new result
of this paper and its proof consists of combining results of  \cite{FFKM} and \cite{L15}. We hope to be able to obtain
a new proof, independent of \cite{FFKM}, in a forthcoming paper.

All in all, the aim of this paper is to convince the reader that it is extremely natural to consider  semi-infinite 
structures in representation theory and that a theory of sheaves on semi-infinite moment graphs  will have applications 
in different branches of representation theory. This is further motivated by results in \cite{BFGM}, \cite{ABBGM}, where 
the geometry of the spaces of quasi maps and of the semi-infinite flag manifold are used to study Lie algebras in 
positive characteristic and quantum groups at a root of unity.

\section{Bruhat order VS semi-infinite order}\label{Sec_BruhatVSSemiInfOrders}

Let $(W,{\mathcal S})$ be a Coxeter system. Every element $w$ of $W$ can be thought of as a (non-unique) word in the alphabet ${\mathcal S}$ and its
length $\ell(w)$ is the minimum number of letters necessary for writing such a word. 
Recall that the set of reflections of $W$ is ${\mathcal T}=\{wsw^{-1}\mid w\in W, \ s\in {\mathcal S} \}$ and the Bruhat order 
$\leq$ on $W$ is the partial order generated by the relations $w\leq t w$ if $\ell(w)<\ell(tw)$ for a $t\in {\mathcal T}$.

Since in this paper we are mainly interested in the representation theoretic side of the story, we will focus on the 
case of $W$ being the Weyl group of an affine Kac-Moody algebra. More precisely, let $\mathfrak{g}$ be a simple finite
dimensional complex Lie algebra, then we consider its affinisation ${\widehat{\mathfrak g}}$ . As a vector space, 
\[{\widehat{\mathfrak g}} \simeq {\mathfrak g}\otimes\mathbb{C}[t^{\pm 1}]\oplus \mathbb{C}K\oplus\mathbb{C}D\]
We are not going to recall the Lie algebra structure on ${\mathfrak g}$, which can be found, for example, in \cite{Kac}, but we 
limit ourselves to mention that $K$ is the central element and $D$ the derivation operator.

Recall that the Weyl group of ${\widehat{\mathfrak g}}$ together with the set of reflections indexed by simple affine reflections is a 
Coxeter system, so that we can consider $W$ endowed with the structure of a poset with respect to the Bruhat order. 

There is a further partial order we want to equip $W$ with, but in order to introduce it, we need to recall the alcove
picture. Let ${\mathfrak h}$ be a Cartan subalgebra of ${\mathfrak g}$, let  $R\subset {\mathfrak h}^*$, resp. $R^\vee\subset {\mathfrak h}$,  be the root, resp.
coroot,  system of ${\mathfrak g}$ and let $\Lambda=\mathbb{Z}R\subset {\mathfrak h}^*$, resp.  $\Lambda^\vee=\mathbb{Z}R^\vee\subset {\mathfrak h}$,  be its root, resp. coroot, lattice. Consider the Euclidean vector 
space $V:=\Lambda\otimes \mathbb{R}$. The Weyl group of ${\widehat{\mathfrak g}}$ can be hence realised  as the group of affine
transformations of $V$ generated by the reflections across the affine hyperplanes
\[H_{\alpha, n}:=\{ v\in V\mid \langle v, \alpha^\vee\rangle =n\}, \qquad (\alpha\in R^+, \ \  n\in \mathbb{Z})\]
where $\langle \cdot, \cdot \rangle: V\times V^* \rightarrow \mathbb{R}$ denotes the natural pairing, $R^+$ the set 
positive roots, and $\alpha^\vee\in R^\vee$ is the coroot corresponding to $\alpha$ (by abuse of notation, we denote
by $\alpha$, resp. $\alpha^\vee$, also its image in $\Lambda\otimes \mathbb{R}=V$, resp. in 
$\Lambda^\vee\otimes \mathbb{R}=V^*$). Thus the (left) action of $W$ on V is given by
\[ s_{\alpha, n}(v)=v-(\langle v,\alpha^\vee\rangle-n)\alpha. \]

The  connected components of $V\setminus\bigcup_{\alpha, n}H_{\alpha,n}$ are called alcoves and, since  $W$ acts
on the set of alcoves ${\mathscr A}$ freely and transitively,  $W$ and ${\mathscr A}$ are in bijection as sets. In order to make such a
bijection into an identification, we need to choose an alcove and look at its orbit under the $W$-action. Let $A_0^-$ be 
the only alcove which contains the origin in its closure and such that $\langle v,\alpha^\vee \rangle <0$  for a(ny)
$v\in A_0^-$ and all $\alpha\in R^+$. We identify $W$ with ${\mathscr A}$ via $w\mapsto w(A_0^-)$. Let ${\mathcal S}$ be the set of
reflections across the hyperplanes containing the walls of  $A_0^-$. Thus $(W,{\mathcal S})$ is a Coxeter system.

The semi-infinite (Bruhat) order, or Lusztig's generic order  \cite{Lu2}, on ${\mathscr A}$
is the partial order generated by the relations $A\leq_{\frac{\infty}{2}} s_{\alpha,n} A$ if $\langle v,\alpha^\vee \rangle >n$ 
for a(ny) $v\in A$ (where $\alpha\in R^+$ and $n\in \mathbb Z$).

\begin{remark} Observe that, while the poset $(W,\leq)$ has a minimal element (the identity), but not a maximal one,
the poset $({\mathscr A},\leq_\frac{\infty}{2})$ has neither a minimal nor a maximal element and it is hence unbounded in both directions. 
It is moreover clear that the semi-infinite order is stable under root translations, that is $A\leq_\frac{\infty}{2} B$ if and only 
if $A+\gamma\leq_\frac{\infty}{2} B+\gamma$ for any $\gamma\in\Lambda$.
\end{remark}

We conclude this section by recalling that the semi-infinite order can be thought of as a limit at $-\infty$ of the 
Bruhat order. The following lemma makes this statement precise.

\begin{lemma}[cf.{\cite[Claim 4.4]{SoeRep}}]\label{Lem_SeminfLimitBruhat} Let $\gamma\in \Lambda$ be such that 
$\langle \gamma,\alpha^\vee\rangle<0$ for any $\alpha\in R^+$ and let $A,B\in{\mathscr A}$. There exists a non-negative integer 
$n_0=n_0(A,B,\gamma)$
 such that the following are equivalent:
\begin{enumerate}
\item[(1)] $A\leq_\frac{\infty}{2} B$
\item[(2)] $A+n\gamma\leq B+n\gamma$ for all $n\geq n_0$. 
\end{enumerate}
\end{lemma}

\begin{remark}The reader should be aware that we defined the semi-infinite order using a convention which is opposite 
to the one in \cite{Lu2}and \cite{SoeRep}, but which agrees with the one in \cite{A}.
\end{remark}

\section{Inclusions of orbit closures}

Recall that the Bruhat order on Weyl groups comes from geometry. We briefly remind such a fact  in the case of an affine
Weyl group (see, for example, \cite{IM}).

 Let $G$ be the connected, simply connected, complex algebraic group with root system $R$. For any subgroup $M$ of
 $G$, denote by $M((t))$, resp. $M[[t]]$, the group  of $\mathbb{C}((t))$-, resp. $\mathbb{C}$[[t]],-points of $M$. For
 example, $G((t))$  is the loop group associated with $G$.  The Iwahori subgroup $I$ of $G((t))$ is the preimage of the 
 Borel subgroup $B\subseteq G$ (corresponding to $R^{+}$) under the map $G[[t]]\rightarrow G$ given by
 $t\mapsto 0$. The affine flag variety $\mathcal F l=G((t))/I$  is naturally equipped with a left action of $I$. The $I$-orbits 
 give the Iwahori  decomposition 
\[\mathcal F l= G((t))/I = \bigsqcup_{w\in W} X_w, \]

 where $X_w=IwI/I\simeq \mathbb{C}^{\ell(w)}$ is called Bruhat cell and its closure $\overline{X_w}$
is a Schubert variety, which turns out to be the union of the Schubert cells indexed by the elements which are 
less or equal than $w$ in the Bruhat order:
\[\overline{X_w}=\bigsqcup_{y\leq w} X_y.\]

Therefore the affine flag variety is an honest ind-scheme, stratified by (finite dimensional) $I$-orbits, whose
ind-structure is given by the Schubert varieties, which are in fact schemes of finite type.

In order to discuss the semi-infinite setting we first need to introduce some notation.  Let $N$ be the unipotent 
radical of $B$ and $T$ the maximal torus of $G$ such that $B=NT$. The semi-infinite flag variety, as introduced in 
\cite{Lu1} and \cite{FF}, is \[\mathcal F l^\frac{\infty}{2}= G((t))/N((t))T[[t]]. \]

Again, the Iwahori acts naturally on $\mathcal F l^\frac{\infty}{2}$ and its orbits (the \emph{semi-infinite Schubert cells}) are in 
bijection with the affine Weyl group. Unluckily, the space above is  not a good algebro geometric object, since it 
cannot be realised as an ind-scheme, as in the case of the affine flag variety: one could hope to play the same game 
as before  and get the ind-structure on  the semi-infinite flag variety from closures of $I$-orbits on it, but the
$I$-orbits have now  infinite dimension and  infinite codimension in $\mathcal F l^\frac{\infty}{2}$. Therefore  it is not even clear 
how to rigorously  define $I$-orbit closures in this setting and hence  \emph{semi-infinite Schubert varieties} 
$\overline{X^\frac{\infty}{2}_w}$. Anyway, if the closure $\overline{X^\frac{\infty}{2}_w}$ made sense, it should be equal to the
union of the $\overline{X^\frac{\infty}{2}_y}$ with $y\leq_\frac{\infty}{2} w$. In fact, in \cite[Section 5]{FFKM} it is proven that the adjacency
order on the set of Schubert strata  in the quasi map spaces (see\cite{FFKM} for the definition) is equivalent to the semi-infinite order. 
We conclude that the semi-infinite order should come from the inclusion relation of semi-infinite Schubert 
varieties (once made sense of them), in the same way as the Bruhat order coincides with the order relation given by the 
inclusion of closures of $I$-orbits  on the affine flag variety.

\section{Torus actions and moment graphs}\label{Sec_ToriAndMG}

Let $Y$ be a lattice, i.e. a free abelian group, of finite rank. A \emph{moment graph on $Y$} is a graph whose edges are labelled by non-zero elements
of $Y$ and whose set of vertices is equipped with the structure of a poset such that two vertices are comparable if 
they are connected. We assume moreover that there are no loops. 

Let us keep the same notation as in the previous sections. 
The extended torus $\widehat{T}:=T\times \mathbb C^\times$ acts on the affine flag variety 
$\mathcal F l$: $T$ by left multiplication on $G$ and $\mathbb C^\times$ by ``rotating the loop" (i.e., by rescaling the
variable $t$) and we can consider the 1-skeleton of such an action (cf. \cite[Chapter 7]{Ku}). The fixed point set 
consists of isolated points and it is in bijection with the affine Weyl group $W$ (each Schubert cell contains exactly
one fixed point), so that from now on we will identify $\mathcal F l^{\widehat T}$ with $W$. Moreover, the closure of any
1-dimensional orbit is smooth and contains exactly two fixed points. Therefore the set of $0$- and 1-dimensional obits
of the $\widehat{T}$-action gives us a graph.  The closure of each 1-dimensional orbit $\mathcal{O}\simeq \mathbb C^\times$ is a one-dimensional representation
of the extended torus and we can hence label the edge $\mathcal O$  by the corresponding character. 

Let us consider the finite dimensional Lie algebra $\mathfrak{g}:=\textrm{Lie}(G)$ and its Cartan subalgebra $\mathfrak{h}:=\textrm{Lie}(T)$, and denote, as in Section 2,  by $\widehat{\mathfrak{g}}$ its affinisation.  Let us write $\widehat{\mathfrak{h}}^*$ for the dual of the affinisation of $\mathfrak{h}$.
Recall that there exists an element $\delta\in\widehat{\mathfrak{h}}^*$ such that the set of real roots of the affine Kac-Moody group associated to $G$ (which is a central extension of $G((t))$) is equal to $\{\alpha+n\delta\mid \alpha\in R,\,n\in \mathbb{Z}\}$.
By  \cite[Chapter 7]{Ku},  two fixed points $x\neq y$ are in the closure of the same 1-dimensional orbit
$\mathcal O$ if and only if there exist $\alpha\in R^+$ and $n\in \mathbb Z$ such that $x=s_{\alpha,n}y$ and 
$\widehat T$ acts on $\overline{\mathcal O}$ via $\pm(\alpha+n\delta)$. Notice that the label of an edge is 
well-defined up to a sign. Anyway, once the sign is fixed, nothing will depend on this choice. Observe that if two
vertices of the graph are connected, then they are comparable with respect to the Bruhat order on the affine Weyl group. 
The \emph{affine Bruhat graph} $\mathcal G$ is the graph on $\mathbb Z\widehat R:=\mathbb{Z}R\oplus \mathbb{Z}\delta$ given by the previous data, together with the Bruhat
order on its set of vertices. Each full subgraph $\mathcal G_w$ of $\mathcal G$ having as a set of vertices the elements which are 
less or equal than a given $w\in W$ coincides with the moment graph associated by Braden and MacPherson  in \cite{BM}
to the Schubert variety $\overline{X_w}$, as a stratified variety with a \emph{nice enough}  $\widehat T$-action.

The extended torus $\widehat T$ acts on the semi-infinite flag variety too. The set of fixed points is once again 
in 1-1 correspondence with the set of $I$-orbits and, hence, we will identify it with $W$. Even if the geometric side 
is not yet rigorously defined or understood, we will pretend to be able to associate with the $\widehat T$-action on 
$\mathcal F l^\frac{\infty}{2}$ a moment graph on $\mathbb Z\widehat R$ . We hope to be able to make this construction more natural, by using
Drinfeld's spaces of quasi maps (see, for example, \cite{Br}), in a forthcoming paper. We hence define the
\emph{semi-infinite moment graph} $\mathcal G^\frac{\infty}{2}$ as follows: it has same set of vertices, edges and label function 
as $\mathcal G$, but the structure of poset on its set of vertices is given by the semi-infinite order.

\section{Structure algebras and cohomology of equivariantly formal spaces}\label{Sec_StructureAlgsEqvtCohom}

Let $\mathcal G$ be a moment graph on a lattice $Y$,  $k$  a field and $S$ the symmetric algebra of the $k$-vector space 
$Y\otimes k$.
We will write ${\mathcal V}$, respectively $\mathcal E$, for the set of vertices, respectively  edges, of $\mathcal G$.
For an edge $E\in \mathcal E$ we will denote by $l(E)\in Y$ its label.
 The structure algebra of $\mathcal G$ is 
\[
{\mathcal Z}:=\left\{
(f_x)\in \prod_{{\mathcal V}} S\ \Big| \ 
\begin{array}{c}
                                    f_x- f_y \in l(E) S \\
                                    \text{if }E=(x,y)\in \mathcal E
                                   \end{array}
\right\}.
\]

Notice that $S$ is diagonally embedded in ${\mathcal Z}$ and that ${\mathcal Z}$ is equipped with a structure of $S$-algebra, given by 
componentwise addition and multiplication.

Observe moreover that the definition of $\mathcal G$ is independent of the partial order on the set of vertices of $\mathcal G$.
This has in fact a geometric reason. In the previous section we saw an example of a moment graph arising as the 
1-skeleton of the action of an algebraic torus on a stratified variety. The partial order on the set of vertices was 
coming from the stratification, 
so forgetting about the partial order is somehow equivalent to forgetting about the stratification.
In fact, if the moment graph $\mathcal G$ coincides with the one skeleton of the action of an algebraic torus $T$ on 
a complex \emph{equivariantly formal variety}  $X$ (no stratification needed!), Goresky, Kottwitz and MacPherson
\cite{GKM98} showed that (for $\textrm{char }k=0$)
$H^\bullet_T(X, k)\simeq {\mathcal Z}$ as $\mathbb Z$-graded $S$-modules, where the $\mathbb Z$-grading on ${\mathcal Z}$ is induced by 
the $\mathbb Z$-grading 
on $S$ given by  $\deg Y:=2$. The structure of $S$-module on $H^\bullet_T(X, k)$ is given by the classical
identification of $S$ with $H^\bullet_T(\textrm{pt})$. Goresky, Kottwitz and MacPherson showed also that it is possible
to recover the usual cohomology just by change of base:  $$H^\bullet(X,k)\simeq {\mathcal Z}\otimes_S k.$$
A $T$-space is equivariantly formal if its $T$-equivariant cohomology is free as an $S$-module.
For example, a $T$-variety whose odd cohomology groups all vanish is equivariantly formal.  
So Schubert varieties are an example of equivariantly formal spaces. 

Let $w\in W$ be an element of the affine Weyl group and consider the graph $\mathcal G_w$ of the previous section.
If we denote by ${\mathcal Z}_w$ its structure algebra, then
${\mathcal Z}_w\simeq H^\bullet_{\widehat T}(\overline{X_w})$ and ${\mathcal Z}_w\otimes k \simeq H^\bullet(\overline{X_w})$. 
Taking inductive limits we also get ${\mathcal Z}\simeq H^\bullet_{\widehat T}(\mathcal F l) $, where now ${\mathcal Z}$ denotes the structure 
algebra of the affine Bruhat graph.

Recall that, once forgotten the partial order on the set of vertices, $\mathcal G$ and $\mathcal G^\frac{\infty}{2}$ coincide, so that their 
structure algebras also coincide.  We have not tried to make sense of the $\widehat T$-equivariant cohomology of
$\mathcal F l^\frac{\infty}{2}$ yet and do not know whether the equality of the structure algebras of  $\mathcal G$ and $\mathcal G^\frac{\infty}{2}$ has a 
rigorous geometric interpretation, but it is certainly compatible with  \cite[Proposition 1]{FF}.

\section{Hecke modules}\label{Sec_HeckeMods}

In this section we recall the definition of the Hecke algebra and of certain modules for the action of the affine
Hecke algebra, whose connection with representation theory will be discussed in Section \ref{Sec_Formulae}.
We use Soergel's notation and normalisation \cite{SoeRep}.

Denote by $\mathcal L$ the ring of Laurent polynomials in one variable with integer coefficients $\mathbb Z[v^{\pm 1}]$. The 
Hecke algebra $\mathcal H$ associated with the  Coxeter system $(W,{\mathcal S})$ is the free $\mathcal L$-module with basis $\{H_y\}$ 
indexed by $W$ and whose structure of associative $\mathcal L$-algebra (and hence of right $\mathcal H$-module over itself) is
uniquely determined by 
\begin{equation}\label{Eqn_HActionOnH}
H_y(H_s+v)=
\begin{cases}
H_{ys}+ vH_y&\text{ if } ys>y,\\
H_{ys}+ v^{-1}H_y&\text{ if } ys<y.
\end{cases}
\end{equation}
It follows that, for any simple reflection $s\in {\mathcal S}$, $H_s^2=(v^{-1}-v)H_s+H_e$ and hence $H_s^{-1}=H_s-(v^{-1}-v)$.
Moreover, if $y=s_{i_1}, \ldots s_{i_r}$, where $s_{i_1}, \ldots, s_{i_r}\in \mathcal{S}$ is a reduced expression 
(that is, $r=\ell(y)$), then $H_y=H_{s_{i_1}}\ldots H_{s_{_r}}$ and hence 
$H_y^{-1}=H_{s_r}^{-1}\ldots H_{s_1}^{-1}\in \mathcal H$ for any $y\in W$.
Thus we can define the bar involution  $\overline{ \cdot }:\mathcal H \rightarrow \mathcal H$,
which is the $\mathbb Z$-linear involutive automorphism of the affine Hecke algebra (as a $\mathbb Z$-algebra) given by: 
$v^{\pm 1}\mapsto v^{\mp 1}$ and $H_y\mapsto H_{y^{-1}}^{-1}$. The following is a classical and well-known result by 
Kazhdan and Lusztig. The formulation we give here is not the original one, but can be found, for example, 
in \cite{SoeRep}. 

\begin{theorem}[{\cite{KL79}, \cite[Theorem 2.1]{SoeRep}}]For any $w\in W$ there is a unique element $\underline{H}_w\in \mathcal H$ such that
$\overline{\underline{H}_w}=\underline{H}_w$ and $\underline{H}_w\in H_w +\sum_{y\in W\setminus \{w\}}v\mathbb Z[v] H_y$.
\end{theorem}

The coefficients of the change of basis matrix $(h_{y,w})$ from $\{H_y\}$ to $\{\underline{H}_w\}$ are, by the above result, 
polynomials in $v$ and are called Kazhdan-Lusztig  polynomials. Notice that $H_s+v$ is self dual, so that
$\underline{H}_s=H_s+v$ and hence  \eqref{Eqn_HActionOnH} is in fact describing the right action of $\mathcal H$ on itself in 
terms of  multiplication by $\underline{H}_s$.

As in the previous sections, we now want to focus on the case of $W$ being an affine Weyl group. Let ${\mathcal S}$ denote the 
set of Coxeter generators for $W$ given  in Section \ref{Sec_BruhatVSSemiInfOrders} and  let $\mathcal H$ be the corresponding
Hecke algebra. We will briefly recall from \cite{Lu2} Lusztig's construction of three $\mathcal H$-modules. 

The periodic Hecke module ${\mathcal P}$ is the free $\mathcal L$-module with basis ${\mathscr A}$ and with a  structure of right $\mathcal H$-module
given by
\begin{equation}\label{Eqn_HActionOnP}
A\cdot \underline{H}_s=
\begin{cases}
As+ vA&\text{ if } As>_\frac{\infty}{2} A,\\
As+ v^{-1}A&\text{ if } As<_\frac{\infty}{2} A.
\end{cases}
\end{equation}
Notice that, once replaced the Bruhat order by the semi-infinite order, \eqref{Eqn_HActionOnH} and \eqref{Eqn_HActionOnP}
coincide.

In order to introduce the second $\mathcal H$-module we want to deal with, some more notation is needed. First, recall that in Section 2 
we have denoted by $A_0^-$ the anti-fundamental alcove, that is 
the only alcove which contains the origin in its closure and such that $\langle v,\alpha^\vee \rangle <0$  for a(ny)
$v\in A_0^-$ and all $\alpha\in R^+$. Let $W_0\subset W$ be
the finite Weyl group, which we identify with the stabiliser in $W$ of $0\in V$.  Let $Q \subset V$ be the set of 
integral weights, that is $Q=\{v\in V\mid \langle v,\alpha^\vee \rangle \in\mathbb Z \text{ for any }\alpha\in R\}$. For any  $\lambda\in Q$ we set
\[
E_\lambda:=\sum_{x\in W_0} v^{\ell(x)} x(A_0^-)+\lambda
\]
and we denote by ${\mathcal P}^0$ the $\mathcal H$-submodule of ${\mathcal P}$ generated by the set $\{E_\lambda\mid \lambda\in Q\}$.

A map $\psi:\mathcal M\rightarrow \mathcal N$ of $\mathcal H$-modules is called $\mathcal H$-skew linear if $f(m\cdot h)=f(m)\cdot \overline{h}$ for
any $m\in \mathcal M$ and $h\in \mathcal H$. In \cite{Lu2} it is shown that  there exists a unique $\mathcal H$-skew linear involution 
$\overline{ \cdot }:{\mathcal P}^0\rightarrow {\mathcal P}^0$ such that $\overline{E_\lambda}=E_\lambda$ for any $\lambda\in Q$. 

In this setting the analogue of Kazhdan-Lusztig's Theorem is the following result.
\begin{theorem}[\cite{Lu2}]For any $A\in {\mathscr A}$ there is a unique element $\underline{P}_A\in {\mathcal P}^0$ such that
$\overline{\underline{P}_A}=\underline{P}_A$ and $\underline{P}_A\in A+\sum_{B\in {\mathscr A}\setminus \{A\}}v\mathbb Z[v] B$.
\end{theorem}
Moreover, the set $\{\underline{P}_A\mid A\in {\mathscr A}\}$ is an  $\mathcal L$-basis of ${\mathcal P}^0$ and the polynomials 
$p_{B,A}\in v\mathcal L$ defined by $\underline{P}_A=\sum_{B\in {\mathscr A}}p_{B,A}B$ are called \emph{periodic polynomials}.

The definition of the third $\mathcal H$-module we want to deal with needs a sort of completion of $\mathcal P$. 
We say that a map $f:{\mathscr A}\rightarrow \mathcal L$ is  bounded from above if there exists an alcove $C\in {\mathscr A}$ such that whenever
$f(A)\neq 0$, then $A\leq_\frac{\infty}{2} C$. We consider the space of all bounded from above maps:
\[
\widehat {\mathcal P}:=\{f:{\mathscr A}\rightarrow \mathcal L \mid  f \text{ is bounded from above}\}
\]
and we identify it with a set of formal $\mathcal L$-linear combinations of alcoves via $f\mapsto \sum f(A) A$. At this point 
it is clear that it is possible to extend the right action of $\mathcal H$ and the  $\mathcal H$-skew-linear involution on ${\mathcal P}$ to
$\widehat {\mathcal P}$. The second semi-infinite variation of Kazhdan-Lusztig's Theorem is hence the  following:
\begin{theorem}[\cite{Kato}]For any $A\in {\mathscr A}$ there is a unique element $\tilde{\underline{P}}_A\in \widehat{\mathcal P}$ such 
that  $\overline{\tilde{\underline{P}}_A}=\tilde{\underline{P}}_A$ and
$\tilde{\underline{P}}_A\in A+\sum_{B\in {\mathscr A}}v^{-1}\mathbb Z[v^{-1}] B\in\widehat {\mathcal P}$.
\end{theorem}
We are now ready to define the last family of polynomials we are interested in in this paper: 
the \emph{generic polynomials} $q_{B,A}\in\mathbb Z[v]$ are such that $\tilde{\underline{P}}_A = \sum_{B\in {\mathscr A}\setminus \{A\}} \overline{q_{B,A}} B$, where the latter is allowed to be
(and indeed it is) an infinite sum.

\section{Multiplicity formulae}\label{Sec_Formulae}

\subsection{Affine Kac-Moody algebras}Let $\widehat{{\mathfrak g}}$ denote the affinisation of the simple Lie algebra ${\mathfrak g}$ as in Section \ref{Sec_BruhatVSSemiInfOrders} and
recall that $K$ is a fixed generator of the central line and $D$ the derivation operator. 
Let  ${\mathfrak h}\subset {\mathfrak b}$ be a Cartan and a Borel subalgebra of ${\mathfrak g}$ (corresponding to $R^+$) and 
let $\widehat {\mathfrak h}$ and $\widehat {\mathfrak b}$ be the corresponding affine Cartan and Borel subalgebras of $\widehat {\mathfrak g}$ (so, in particular, $\widehat {\mathfrak h}\simeq {\mathfrak h}\oplus\mathbb C K\oplus \mathbb C D$).
The affine BGG category $\mathcal O$ is the full subcategory of $\widehat{\mathfrak g}$-modules on which $\widehat {\mathfrak h}$ acts semi-simply and $\widehat {\mathfrak b}$ locally finitely.
Then the affine BGG category $\mathcal O$  decomposes in levels $\mathcal O_\kappa$, $\kappa\in\mathbb C$, according to 
the action of the central element $K$. Let us fix an element $\rho$ in the dual of the Cartan $\widehat{\mathfrak h}^*$ 
 having the property that $\rho(\gamma^\vee)=1$ for any simple affine root $\gamma$. Observe that $\rho$ is uniquely determined only up to a multiple of the smallest positive imaginary root $\delta$.
This is in fact not a problem, since nothing is going to depend on this choice.  The critical level is the value $\textrm{crit}:=-\rho(K)$.

For any weight $\lambda\in \widehat{\mathfrak h}^*$, we denote by $\Delta(\lambda)$ and $L(\lambda)$ the  Verma module and  
the irreducible representation of highest weight $\lambda$, respectively.

\subsubsection{Negative level} Let us consider a regular weight $\lambda$ with $\lambda(\gamma^\vee)\in\mathbb{Z}_{<-1}$ for any simple affine root $\gamma$. Then we have the following affine version of the Kazhdan-Lusztig conjecture:
\begin{equation}\label{Eqn_NegativeKL}
[\Delta(y\cdot \lambda):L(w\cdot \lambda)]=\tilde{h}_{y,w}(1)
\end{equation}
where, for any $x\in W$, the $\rho$-shifted action of $W$ on $\widehat {\mathfrak h}^*$ is defined as 
$x\cdot \lambda:=x(\lambda+\rho)-\rho$, and $\tilde{h}_{y,w}$ denotes the inverse  Kazhdan-Lusztig polynomial, that is the polynomial such that $\sum_{y}(-1)^{\ell(u)-\ell(y)} h_{u,y}\tilde{h}_{y,w}=\delta_{u,w}$.
The above multiplicity statement was proven by Kashiwara and Tanisaki using the theory of $\mathcal D$-modules on the 
affine flag variety \cite{KasTan}. An alternative, algebraic proof can be deduced from \cite{EW}, where Elias and Williamson
prove a more general statement about indecomposable Soergel bimodules, known as Soergel's conjecture. We will briefly
discuss the moment graph formulation of Soergel's conjecture later in Section \ref{Sec_BMPMultiplicities}.

\subsubsection{Positive level} If a regular weight $\lambda$ is such that $\lambda(\gamma^\vee)\in\mathbb{Z}_{> -1}$ for any simple affine root $\gamma$, then it is enough to substitute $\tilde{h}_{y,w}$ in the multiplicity statement \eqref{Eqn_NegativeKL} by the Kazhdan-Lusztig polynomial $h_{y,w}$. Such a result was proven by Kashiwara and Tanisaki \cite{KasTanPos} using once again the theory
of $\mathcal D$-modules, but in this case, they had to deal with sheaves of differential operators on Kashiwara's thick 
flag variety, whose definition we do not want to recall. It is sufficient for us to mention  that such a variety, as
$\mathcal F l$ and $\mathcal F^\frac{\infty}{2}$, is an appropriate quotient of a loop group, on which the Iwahori acts with orbits of infinite
dimension (but finite codimension). As for the negative level case, an algebraic proof of the affine Kazhdan-Lusztig 
conjecture at a positive level can be deduced from \cite{EW}. 

\subsubsection{Critical level}\label{Subsubsec_CriticalLevel} Finally, we want to look at the case of a dominant
regular $\lambda\in\widehat {\mathfrak h}^*$ such that $\lambda(K)=\textrm{crit}$. 
(Moreover, $\lambda$ should satisfy some technical conditions that we do not want to list and that can be found in 
\cite[\S1.4]{AF}).
We will state a (still in general conjectural) further multiplicity formula involving restricted Verma modules. 
Tensoring by the one-dimensional representation $L(\delta)$ is an autoequivalence of $\mathcal O$ and a block is
stabilised by such an equivalence if and only if it is of critical level, so that tensoring by $L(\delta)$ is also an
auto-equivalence of any critical block. 
In order to define restricted Verma modules the notion of graded centre of a critical block is needed.
Its degree $n$ part consists of the space of natural transformations from the functor $\cdot \otimes L(n\delta)$ to the 
identity functor on 
the block (satisfying certain extra conditions, which are described, for example, in \cite[\S1.4.]{AF}).
Then the restricted Verma module $\overline{\Delta}(\lambda)$ is obtained by quotienting $\Delta(\lambda)$ by the ideal
generated by
the homogeneous components of degree $\neq 0$ of the centre. The following formula has been conjectured 
by Feigin, Frenkel and Lusztig, independently:
\begin{equation}\label{Eqn_CriticalFFL}
[\overline{\Delta}(y\cdot \lambda):L(w\cdot \lambda)]=p_{y(A_0^-),w(A_0^-)}(1).
\end{equation}

\subsubsection{Wakimoto modules} We leave to \cite{AL} the discussion of a multiplicity formula for Wakimoto modules at a
non critical level involving Lusztig's generic polynomials evaluated at one.

\subsection{Quantum groups at a root of unity}\label{Subsec_QuantumGps} Let ${\mathfrak g}$ be again a simple complex finite 
dimensional Lie algebra and 
consider (Lusztig's version of) its quantum group at a $p$-th root of unity \cite{Lu90b}, where $p$ is an odd integer 
(prime to 3 in the $G_2$-case).  Denote by  $\mathbf{u}_p({\mathfrak g})$ the (finite dimensional) small quantum group.  
Then  $\mathbf{u}_p({\mathfrak g})$ admits a triangular decomposition and the standard objects  $\{Z(\mu)\}$ in this 
setting have the same realisation as Verma modules for ${\mathfrak g}$: they are obtained by inflating a linear form of the 
0-part to the positive part and then inducing it to the whole $\mathbf{u}_p({\mathfrak g})$. The module $Z(\mu)$ has a unique 
simple quotient, that we denote $L(\mu)$. 

Let $p$ be greater than the Coxeter number of ${\mathfrak g}$, and $\lambda\in {\mathfrak h}^*$ such that 
$\langle\lambda, \alpha^\vee\rangle<-1$ for any simple root $\alpha$ of ${\mathfrak g}$ and 
$\langle\lambda,\varphi^{\vee}\rangle> p-1$, where $\varphi$ is the highest root of ${\mathfrak g}$. Denore by $\rho_0$ half the sum of all positive roots. We consider the $p$-dilated 
$\rho_0$-shifted action $\cdot_p$ of the affine Weyl group on ${\mathfrak h}^*$:  this means that we shift by $-\rho_0$ the action of
the affine Weyl group $W$ which has been rescaled in such a way that 
$s_{\alpha,n}(v)=v-(\langle v,\alpha^\vee\rangle-pn)\alpha$. Then the following multiplicity formula is a restatement of
Lusztig's conjecture on type 1 finite dimensional modules for quantum groups at a root of unity (cf. \cite{Lu90b}):
\begin{equation}\label{Eqn_QuantumGpsRoots1}
[Z(y\cdot_p \lambda):L(w\cdot_p \lambda)]=p_{y(A_0^-),w(A_0^-)}(1),
\end{equation}
where $y\cdot \lambda$ is dominant and  $w$ is a minimal element in $w\textrm{Stab}_{W\cdot_p}(\lambda)$.

\subsection{$G_1T$-modules}\label{Subsec_G1TMods}Let $k$ be a field of characteristic $p$,  $G$  a semisimple simply
connected algebraic group over $k$ and  $G_1\subset G$ be the kernel of the Frobenius.  Fix once and for all a Borel 
subgroup B, a maximal torus  $T\subseteq B$, and denote by ${\mathfrak b}$ and ${\mathfrak h}$ the corresponding Lie algebras. The Lie
algebra ${\mathfrak g}$ of $G$ is a $p$-algebra and we denote by $U^\textrm{res}({\mathfrak g})$ its restricted Lie algebra, which is a 
finite dimensional quotient of the enveloping algebra $U({\mathfrak g})$ (see, for example \cite[Introduction]{AJS}). Moreover, 
let $G_1T\subset G$ be the group scheme generated by $T$ and $G_1$. The representation category we are interested in has 
as objects the finite dimensional $G_1T$-representations, which can be identified with $U^\textrm{res}({\mathfrak g})$-modules 
graded by the group of characters $X$ of $T$. Let $\mu\in X$, and differentiate it to get $\overline{\mu}\in{\mathfrak h}^*$. The 
Baby Verma module $Z(\mu)$ is defined as the induction to the whole $U^\textrm{res}({\mathfrak g})$ of the $X$-graded 
one-dimensional $U({\mathfrak b})$-representation concentrated in degree $\mu$ (which is obtained, as usual, by inflating to 
$U({\mathfrak b})$ the one-dimensional ${\mathfrak h}$-module $k_{\overline{\mu}}$). Denote by $L(\mu)$ the unique simple quotient of 
$Z(\mu)$. As in \S \ref{Subsec_QuantumGps},  let $W$ be the affinisation of the Weyl group of $G$ and, as in \S\ref{Subsec_QuantumGps},  denote by $\cdot$ 
the $\rho_0$-shifted  action of it on ${\mathfrak h}^*$. Then the $G_1T$-version of Lusztig's conjecture on the characters of 
modular representations (of regular restricted highest weights) is the following (cf. \cite[Conjecture 3.4]{Fie07})
\begin{equation}\label{Eqn_G1TModules}
[Z(y\cdot 0):L(w\cdot 0)]=p_{y(A_0^-),w(A_0^-)}(1),
\end{equation}
for $w,y\in W$ and $w$ such that $-p<\langle w\cdot 0, \alpha^\vee\rangle \leq 0 $. For $p\gg 0$, in \cite{AJS} the 
above multiplicity formula is derived from the quantum group at a root of unity analogue, and  an explicit huge bound on
$p$, depending on the root system of $G$, for the statement to be true was found by Fiebig in \cite{Fie12}. Till
June 2013, when Williamson announced the first counterexample \cite{W},  Lusztig's conjecture was expected to hold for 
$p$ greater than the Coxeter number $h$ of $G$ (this was Kato's hope \cite{Kato}, even more optimistic than Lusztig's 
suggestion of $p\geq 2h-3$). A modified conjecture is not available yet, so that new tools are now needed.


\section{Sheaves on moment graphs} Let $\mathcal G$ be a moment graph on $Y$ and $S$ the symmetric algebra of $Y\otimes k$ for
a given field $k$, $\mathbb Z$-graded as in Section \ref{Sec_StructureAlgsEqvtCohom}. Denote by 
$S-\textrm{mods}^{\mathbb Z}$ the category of $\mathbb Z$-graded $S$-modules.  A sheaf $\mathscr F$ on $\mathcal G$ is given by two
collections of $\mathbb Z$-graded $S$-modules
\[
( \mathscr F^x\in S-\textrm{mods}^{\mathbb Z})_{x\in {\mathcal V}}\ , \quad ( \mathscr F^E\in S-\textrm{mods}^\mathbb{Z} \mid l(E)\mathscr F^E=(0) )_{E\in \mathcal E}
\]
and a collection of maps of $\mathbb Z$-graded $S$-modules
\[
(\rho^{x,E}: \mathscr F^x\rightarrow \mathscr F^E)_{\substack{E\in\mathcal E\qquad\qquad\\ x \textrm{ is a vertex of }E}}.
\]
For a vertex $x$, we will often refer to the module $\mathscr F^x$ as the stalk in $x$ of $\mathscr F$. 
A morphism $f$ between two sheaves $\mathscr F_1$ and $\mathscr F_2$ on $\mathcal G$ consists of two collections of morphisms of 
$\mathbb Z$-graded $S$-modules
\[
( f^x:\mathscr F_1^x\rightarrow \mathscr F_2^x)_{x\in {\mathcal V}}\ ,  \quad (f^E: \mathscr F_1^E\rightarrow \mathscr F_2^E )_{E\in \mathcal E}
\]
compatible with the $\rho$-maps, that is $f^E\circ\rho_1^{x,E}=\rho_2^{x,E}\circ f^x$ for any edge $E$ and any vertex
$x$ of $E$.

The set of sections of $\mathscr F$ over  $\mathcal I\subseteq {\mathcal V}$,  is 
\[
\Gamma(\mathcal I, \mathscr F):=\left\{
(f_x)\in \prod_{x\in\mathcal I} \mathscr F^x\ \Big| \ 
\begin{array}{c}
                                   \rho^{x,E}(f_x)= \rho^{y,E}(f_y)\\
                                    \text{if }E=(x,y)\in \mathcal E
                                   \end{array}
\right\}.
\]
Notice that $\Gamma(\mathcal I,\mathscr F)$ has naturally a structure of $\mathbb Z$-graded $S$- and ${\mathcal Z}$-module, on which $S$ acts 
diagonally and ${\mathcal Z}$ componentwise.

Observe that the structure algebra can be realised as the set of global sections (i.e, sections over the whole ${\mathcal V}$) of
the sheaf $\mathscr Z$ given by 
\[
( \mathscr Z^x= S)_{x\in {\mathcal V}}\ , \quad ( \mathscr Z^E= S/l(E)S )_{E\in \mathcal E}
\]
and the $\rho$-maps  $\rho^{x,E}:S\rightarrow S/l(E)S$ (for $x$ vertex of $E$)  are the canonical quotient maps.

\subsection{Braden-MacPherson sheaves} We recall here the definition of a class of sheaves on a moment graph introduced by Braden and
MacPherson \cite{BM} in order to compute the intersection cohomology of varieties acted upon by an algebraic torus $T$, 
which are equivariantly formal and equipped with a $T$-stable stratification (see \cite[\S1.1]{BM} for the exact 
assumptions on the varieties). Their construction makes sense also for moment graphs not coming from geometry.

Recall that, by definition, the vertex set of any moment graph $\mathcal G$ is equipped with a partial order $\leq$. For a vertex
$x\in {\mathcal V}$, denote by $\{>x\}$ the set of elements which are strictly greater than $x$. Let $\mathscr F$ be a sheaf and call
$\rho^{\delta x}$ the following composition of maps:
\[
\Gamma(\{>x\}, \mathscr F)\hookrightarrow \bigoplus_{y>x}\mathscr F^y\twoheadrightarrow  \bigoplus_{\substack{y>x \textrm{ s.t.}\\ E=(x,y)\in\mathcal E}}\mathscr F^y \stackrel{\oplus \rho^{y,E}}{\longrightarrow}  \bigoplus_{\substack{ E=(x,y)\in\mathcal E \\\textrm{s.t. } y>x }}\mathscr F^E.
\]
We denote by $\mathscr F^{\delta x}$ the $\mathbb Z$-graded $S$-module $\rho^{\delta x}(\Gamma(\{>x\}, \mathscr F))$. 

For any $w\in {\mathcal V}$, the indecomposable Braden-MacPherson sheaf  $\mathscr B(w)$ is inductively defined as follows. We 
start by setting:
\[
\mathscr{B}(w)^y=(0) \ \hbox{if $y\not\leq w$}\ , \qquad \mathscr{B}(w)^w\simeq S,
\]
\[
\mathscr{B}(w)^{(y,z)}=(0)\hbox{ and }  \rho^{y,(y,z)}=\rho^{z,(y,z)}=0  \ \hbox{ if $y \not\leq w$ or $z\not \leq w$}
\]
and we then assume that $\mathscr{B}(w)^y$ and $\mathscr{B}(w)^E$ have been constructed for any vertex $y>x$ and 
$E=(y,z)\in \mathcal E$ with $y,z>x$. We define 
\[\mathscr{B}(w)^{E}=\mathscr{B}(w)^y/l(E)\mathscr{B}(w)^y \qquad\hbox{if }E=(x,y) \hbox{ and }y>x 
\]
and, for such an $E$, the morphism  $\rho^{y,E}: \mathscr{B}(w)^y\rightarrow \mathscr{B}(w)^y/l(E)\mathscr{B}(w)^y$ is 
the canonical quotient map. Now it is possible to consider the module $\mathscr{B}(w)^{\delta x}$ and to take its
projective cover (which always exists in the category of $\mathbb Z$-graded $S$-modules):
\[
p_x:\mathscr{B}(w)^x\rightarrow \mathscr{B}(w)^{\delta x}\  \hbox{is a projective cover}.
\]
Obviously, for $(x,z)\in E$ and $z>x$, one obtains the map $\rho^{x,(x,z)}$ as the composition of the following
morphisms
\[
\mathscr{B}(w)^x\stackrel{p_x}{\rightarrow} \mathscr{B}(w)^{\delta x}\hookrightarrow   \bigoplus_{\substack{ E=(x,y)\in\mathcal E \\\textrm{s.t. } y>x }}\mathscr B^E \twoheadrightarrow \mathscr B^{(x,z)}.
\]

\section{Braden-MacPherson sheaves and intersection cohomology} 

As already mentioned, Braden-MacPherson sheaves were introduced with the aim of providing a combinatorial algorithm to
compute equivariant intersection cohomology, with coefficients in a field $k$ of characteristic zero,  of a sufficiently nice complex algebraic variety $X$ equipped with the 
action of an algebraic torus $T$. More precisely, with any such a variety one can associate a moment graph $\mathcal G$ on the  character lattice $Y$ of $T$, whose set of vertices is given by the $T$-fixed points $X^T$ (see \cite[\S1.2]{BM}). The
order on the set of vertices is induced by a fixed stratification on $X$ having the property that any stratum contains 
exactly one fixed point and hence there is a unique maximal vertex $\bar v$. Let  $S$ be the symmetric algebra of the vector space $Y\otimes k$. If $X$ satisfies all the assumptions in \cite[\S1.1]{BM}, then there are canonical
identifications: (cf \cite[Theorem 1.5  and Theorem 1.6]{BM}):
\[
IH^\bullet_T(X)\simeq \Gamma({\mathcal V}, \mathscr B(\bar v))  \qquad \hbox{as ${\mathcal Z}$-modules}, 
\]
and hence also as $S$-modules, and, for each point $x\in X$,
\[
IH^\bullet_T(X)_x\simeq  \mathscr B(\bar v)^y  \qquad \hbox{as $S$-modules}, 
\]
where $y\in X^T$ is the fixed point contained in the same stratum as $x$. As the previous modules are all free over $S$, 
non-equivariant global and local intersection cohomology are obtained by base change:
\[
IH^\bullet(X)\simeq \Gamma({\mathcal V}, \mathscr B(\bar v))\otimes_S k\ ,\qquad  IH^\bullet_T(X)_y\simeq  \mathscr B(\bar v)^x\otimes_S k
\]
as $k$-vector spaces.

For example, any affine Schubert variety $\overline{X_w}$ satisfies the assumptions in \cite[\S1.1.]{BM} and we can
consider the moment graph $\mathcal G_w$ as in Section \ref{Sec_ToriAndMG}. It has a unique maximal vertex, $w$, and hence 
\[IH^\bullet_{\widehat T}(\overline{X_w})\simeq \Gamma({\mathcal V}, \mathscr B(w)) \ ,\hbox{ and } \qquad IH^\bullet_{\widehat T}(\overline{X_w})_y\simeq  \mathscr B(w)^y \hbox{ for any $y\leq w$}\]

In positive characteristic (under some technical assumption on $k$), Braden-MacPherson sheaves compute hypercohomology
of parity complexes on $X$ \cite{FieW}. Parity complexes are not perverse  in general  and this fact is related to the 
presence of torsion in the intersection cohomology groups of Schubert varieties. The discovery of these torsion 
phenomena  allowed Williamson to produce many counter-examples to Lusztig's modular conjecture \cite{W}. 

\section{Local intersection cohomology of $\mathcal F l^\frac{\infty}{2}$} 
Let $M$ be a finitely generated,  free $\mathbb Z$-graded $S$-module, then there are integers $j_1, \ldots j_r$, 
uniquely determined up to reordering, such that $M=\bigoplus^r_{i=1} S[j_i]$ (where the shift in the grading is such 
that $M[j]_{n}=M_{j+n}$). With such an $M$ we can associate its graded rank, that is a Laurent polynomial in $v$ which
keeps track of the shifts: $\underline{\textrm{rk}}M:=\sum^r_{i=1} v^{-j_i}\in\mathbb Z[v^{\pm 1}]$. Let $\mathcal G$ be a moment graph. Observe that for any pair of vertices  $w, y\in{\mathcal V}$ the stalk $\mathscr B(w)^y$ is by 
construction finitely generated and free as a $\mathbb Z$-graded  $S$-module, so that it makes sense to consider its graded rank.

Let $\textrm{char} k=0$. In \cite{L15} we investigated the stable moment graph $\mathcal G^\textrm{stab}$, which is the
full subgraph of $\mathcal G^\frac{\infty}{2}$ having as set of vertices the set ${\mathscr A}^-$ of alcoves such that 
if $v\in A\in {\mathscr A}^-$ then $\langle v, \alpha^\vee\rangle<0$ for any simple finite root $\alpha$.
(To be precise, in \cite{L15} we were dealing with the upside down setting, since we had defined the semi-infinite
order by giving the affine hyperplanes an orientation which is opposite to the one considered here). For any alcove 
$A\in {\mathscr A}^-$, denote by $\mathscr B^{\textrm{stab}}(A)$ the corresponding indecomposable Braden-MacPherson sheaf,
where the ``$\textrm{stab}$'' is there to remind us that we are considering sheaves on $\mathcal G^\textrm{stab}$. 

Let $\delta(A,B)$ be Lusztig's semi-infinite length function (see \cite{Lu2}). The main result of \cite{L15} is the following:
\begin{theorem}\label{Thm_La15} Assume that $A, B\in {\mathscr A}^-$ are \emph{deep enough} in ${\mathscr A}^-$, then
\[\underline{\textrm{rk}}\ \mathscr B^\textrm{stab}(A)^B=v^{\delta(A,B)} q_{B,A}.\]
\end{theorem}

Next, consider the semi-infinite graph $\mathcal G^\frac{\infty}{2}$ and the Braden-MacPherson sheaves on it. Since for any $A\in {\mathscr A}$
the set of elements less than $A$ with respect to $<_\frac{\infty}{2}$ is infinite, the algorithm described in the previous 
section will never end. Nevertheless, all intervals $[B,A]:=\{C\in W\mid B\leq_\frac{\infty}{2} C \leq_\frac{\infty}{2} A\}$ have finite cardinality and hence it is, in principle,  possible to compute $\mathscr B^{\frac{\infty}{2}}(A)^B$, where, again, $\frac{\infty}{2}$ reminds us that we are dealing with sheaves on $\mathcal G^\frac{\infty}{2}$.  Taking their graded rank returns again Lusztig's generic polynomials.

\begin{lemma}\label{Lem_SeminfBMPMult}
$\underline{\textrm{rk}}\ \mathscr B^\frac{\infty}{2}(A)^B=v^{\delta(A,B)} q_{B,A}.$
\end{lemma}
\begin{proof}If $A,B\in{\mathscr A}^-$ and they are both \emph{deep enough} in ${\mathscr A}^-$, then the statement coincides with the statement 
of  Theorem \ref{Thm_La15}.  Otherwise, there exists an antidominant $\gamma$  such that $[B+\gamma,A+\gamma]\subset {\mathscr A}^-$ and $A,B$ are \emph{deep enough} in ${\mathscr A}^-$. For any two alcoves $C,D\in {\mathscr A}$, denote by $\mathcal G^\frac{\infty}{2}_{|[C,D]}$ the full subgraph of $\mathcal G^{\frac{\infty}{2}}$ having as set of vertices the interval $[C,D]$.  There is an induced morphism of moment graphs $T_{\gamma}:\mathcal G^\frac{\infty}{2}_{|[A,B]}\rightarrow \mathcal G^\frac{\infty}{2}_{|[A+\gamma,B+\gamma]}$ which is in fact an isomorphism in the sense of \cite{L12}. By \cite[Lemma 5.1]{L12}, $\mathscr B^\frac{\infty}{2}(A)^B\cong \mathscr B^\frac{\infty}{2}(A+\gamma)^{B+\gamma}$ and now the lemma follows from Theorem \ref{Thm_La15}.
\end{proof}

We want to relate all of this to the  geometry of the semi-infinite flag variety. 
In \cite{FFKM}  the singularities of $\mathcal F l^\frac{\infty}{2}$ are investigated via the study of Drinfeld's 
spaces of quasi maps. For any $\gamma\in \mathbb Z R^\vee$, let us denote by ${\tt Q}^{\gamma}$ the space of 
quasi maps of degree $\gamma$ from a curve $\mathcal C$ of genus zero to the flag variety $G/B$ 
(where $R^\vee$ is the set of coroots of $G$).
 We are not going to recall the definition 
of these objects which can be found in, 
for example,  \cite{FFKM}. Let $w,y\in W=\mathbb Z R^\vee \rtimes W_0$, then $w=t_{\lambda}u$ and $y=t_{\mu}v$, for $\lambda, \mu\in\mathbb Z R^\vee$ 
and $u,v\in W_0$, then we set
\[
IH^\bullet\Big(\overline{X^\frac{\infty}{2}_w}\Big)_y:=IH^\bullet({\tt Q}^{-\lambda}_{u})_{-\mu, v}\ ,
\]
where $IH^\bullet({\tt Q}^{-\lambda}_{u})_{-\mu, v}$ denotes the stalk of 
$IH({\tt Q}^{-\lambda}_{u})$ at the generic point of ${\tt Q}^{-\mu}_{v}$ 
(see \cite[6.4.4.]{FFKM} for the definition of the closed variety ${\tt Q}^{\eta}_{z}$). Our choice of orientation
of the affine hyperplanes, opposite to the one in \cite{FFKM}, requires the ``-" sign in front of $\mu$ and $\lambda$.

The following result had been anticipated in \cite[\S11]{Lu1} and proved in \cite{FFKM}: 
\begin{theorem}
For any pair $y,w\in W$
\[
\sum_i \text{\emph{dim} } IH^{2i}\Big(\overline{X^\frac{\infty}{2}_w}\Big)_y v^i= v^{\delta(y(A_0^-),w(A_0^-))}q_{y(A_0^-), w(A_0^-)}.\]
\end{theorem}

Moreover, by the considerations in \cite[\S3.2]{FFKM} the equivariant cohomology
$IH_{\widehat T}^\bullet({\tt Q}^{-\lambda}_{u})_{-\mu, v}$ is isomorphic, as an $S$-module, to 
$S\otimes IH^\bullet({\tt Q}^{-\lambda}_{u})_{-\mu, v}$, so that we have
\[
\underline{\textrm{rk}}\ IH_{\widehat T}^\bullet({\tt Q}^{-\lambda}_{u})_{-\mu, v}= v^{\delta(y(A_0^-),w(A_0^-))}q_{y(A_0^-), w(A_0^-)}.
\]

By combining the results of this section, we get:
\begin{corollary}The Braden-MacPherson algorithm computes the local cohomology of the semi-infinite flag variety, 
that is for any pair  $y,w\in W$
\[
IH_{\widehat T}^\bullet\Big(\overline{X^\frac{\infty}{2}_w}\Big)_y\simeq \mathscr B(w(A_0^-))^{y(A_0^-)},
\]
\[
IH^\bullet\Big(\overline{X^\frac{\infty}{2}_w}\Big)_y\simeq \mathscr B(w(A_0^-))^{y(A_0^-)}\otimes k.
\]
\end{corollary}

\begin{remark}Notice that the above isomorphisms are obtained as abstract isomorphisms of graded $S$-modules and it is natural to ask whether they have in fact a functorial origin. We hope to be able to give a positive answer to such a question in a forthcoming paper, by investigating the $\widehat T$-equivariant intersection cohomology of quasi-map spaces.
\end{remark}

\section{Stalks of Braden-MacPherson sheaves and multiplicity formulae}\label{Sec_BMPMultiplicities}
Assume $\mathcal G$ is the affine Bruhat  moment graph of Section \ref{Sec_ToriAndMG} and let $k=\mathbb C$. The vertices of
$\mathcal G$ are indexed by elements of an affine Weyl group $W$ and for any  pair $y,w\in W$ 
\begin{equation}\label{Eqn_MultiplicityConj}
\underline{\textrm{rk}}\mathscr B(w)^y=v^{l(y)-l(w)} h_{y,w}.
\end{equation}
The above equality follows from \cite{BM} and the fact that affine Kazhdan-Lusztig polynomials are (up to a shift in our 
normalisation) Poincar\'e polynomials for the stalks of local intersection cohomology of the affine flag variety with 
coefficients in a field  of characteristic zero \cite{KL80}. It is in fact possible to associate with any reflection faithful representation of a Coxeter system
a moment graph and to study indecomposable Braden-MacPherson sheaves on it. The moment graph formulation of Soergel's
conjecture states that \eqref{Eqn_MultiplicityConj} holds in this more general setting too and has been proven, as we 
have  already mentioned, by Elias and Williamson \cite{EW}.

 In the affine Bruhat graph case (and, more in general, for a moment graph associated with a symmetrisable Kac-Moody 
 algebra), Fiebig \cite{Fie} proved a moment graph localisation theorem for modules in (a deformed  version of) category
 $\mathcal O$  admitting a Verma flag, that is a finite filtration with subquotients isomorphic to Verma modules.
 The fundamental application of his result is that one can use Braden-MacPherson sheaves on $\mathcal G$ to compute
 multiplicities of simples in Verma modules at a negative level without passing through geometry:
\begin{equation} 
\label{Eqn_MultiplicityBMPVerma}
\underline{\textrm{rk}}\mathscr B(w)^y(1) = [\Delta(y\cdot \lambda):L(w\cdot \lambda)].
\end{equation}

Once we know \eqref{Eqn_MultiplicityBMPVerma}, it is clear  that from  \eqref{Eqn_MultiplicityConj}   follows \eqref{Eqn_NegativeKL},  the affine version of the Kazhdan-Lusztig conjecture at a negative level for the affinisation of a simple complex Lie algebra ${\mathfrak g}$. 

Next, we want to focus on the semi-infinite graph $\mathcal G^\frac{\infty}{2}$.  In \cite{AL} we prove a (semi-infinite) moment graph
localisation theorem for modules in (a deformed  and truncated version of) the affine BGG category $\mathcal O$  admitting a 
Wakimoto flag, that is a finite filtration with subquotients isomorphic to Wakimoto modules and we expect to be able to
use Theorem \ref{Thm_La15} to interpret multiplicities of simple quotients of Wakimoto modules in terms of generic
polynomials.

We conclude this section by briefly mentioning the main result of \cite{FieLa}. In \cite{FieLa15b} we introduce  the notion of group actions on a moment graph and  in \cite{FieLa} we consider the special case of the root lattice acting on  $\mathcal G^\frac{\infty}{2}$. This allows us to 
define a sort of pushforward functor from the category of sheaves on $\mathcal G^\frac{\infty}{2}$ to a certain category of modules over 
the structure algebra of the quotient graph, equipped with a particularly nice filtration (see \cite{FieLa}). Under this
pushforward functor, the indecomposable Braden-MacPherson sheaves $\mathscr B^\frac{\infty}{2}(A) $ decompose and looking at the
indecomposable summands enables us to define a certain object $\mathscr P(A)$ for any $A\in{\mathscr A}$. For any $B,A\in {\mathscr A}$,
it is possible to consider a finitely generated, free $\mathbb Z$-graded $S$-module $\mathscr P(A)_{[B]}$ and we have:
\begin{theorem}\label{Thm_FieLa15}
$\underline{\textrm{rk}}\ \mathscr P(A)_{[B]}= p_{B,A}.$
\end{theorem}
The above theorem  suggests that our construction should have applications in the questions discussed in Sections 
\ref{Subsec_G1TMods} and \ref{Subsubsec_CriticalLevel}, where Lusztig's periodic polynomials appear.

\section{Acknowledgements}
Many thanks go to  Michael Finkelberg for interesting correspondence, making me realise that the semi-infinite world is much more complicated than what I was hoping, and to Peter Fiebig and Stephen Griffeth for their helpful remarks on a preliminary version of this paper.
 I wish to thank Mainheim Caf\'e in Nuremberg, where this paper has been entirely written, for providing very pleasant working environment. This work was supported by the DFG grant SP1388.

\end{document}